\documentclass[12pt]{article}

\usepackage{a4wide}
\usepackage{amssymb,amsmath,array,amsthm, amsfonts, amscd}
\usepackage{colortbl}

\usepackage{algpseudocode}
\usepackage{algorithm}
\usepackage{tikz}

\allowdisplaybreaks[1]

\newcommand{\titre}{Efficient recognition of totally nonnegative matrix cells}

\newtheorem{theorem}{Theorem}[section]
\newtheorem{corollary}[theorem]{Corollary}
\newtheorem{definition}[theorem]{Definition}

\newtheorem{lemma}[theorem]{Lemma}
\newtheorem{proposition}[theorem]{Proposition}

\theoremstyle{definition}
\newtheorem{mydef}[theorem]{Definition}
\newtheorem{example}[theorem]{Example}
\newtheorem{defi}[theorem]{Definition}
\newtheorem{nota}[theorem]{Notation}
\newtheorem{conv}[theorem]{Convention}

\newtheorem{remark}[theorem]{Remark}

\newcommand{\oc}{\mathcal{O}}

\newcommand{\mc}{\mathcal{M}}
\newcommand{\bbN}{\mathbb{N}}
\newcommand{\bbR}{\mathbb{R}}
\newcommand{\bbC}{\mathbb{C}} 
\newcommand{\pmmpc}{\oc\left( \mathcal{M}_{m,p}(\bbC) \right)}

\newcommand{\ia}{i,\alpha}

\newcommand{\mmptnn}{\mc^{\geq 0}_{m,p}(\bbR)}

\newcommand{\gc}{ [ \hspace{-0.65mm} [}
\newcommand{\dc}{]  \hspace{-0.65mm} ]}

\input xy \xyoption{matrix} \xyoption{arrow}
\def\edge{\ar@{-}}

\def\Ecirc{E^{\circ}}

\long\def\symbolfootnote[#1]#2{\begingroup\def\thefootnote{\fnsymbol{footnote}}\footnote[#1]{#2}\endgroup}

\begin{document}

\title{\titre} \author{S Launois\thanks{The first author is grateful for the
financial support of EPSRC first grant \textit{EP/I018549/1}.}~ and T H
Lenagan}
\date{}

\maketitle


\begin{abstract}\footnotesize 
The space of $m\times p$ totally nonnegative real matrices has a
stratification into totally nonnegative cells. The largest such cell is the
space of totally positive matrices. There is a well-known criterion due to
Gasca and Pe\~na for testing a real matrix for total positivity. This
criterion involves testing $mp$ minors. In contrast, there is no known small
set of minors for testing for total nonnegativity. In this paper, we show that
for each of the totally nonnegative cells there is a test for membership which 
only involves $mp$ minors, thus extending the Gasca and Pe\~na result to 
all totally nonnegative cells. 
\end{abstract}

\vskip .5cm
\noindent
{\em 2010 Mathematics subject classification:} 15B48

\vskip .5cm
\noindent
{\em Key words:} Totally nonnegative matrices; totally nonnegative cells; 
efficient testing criteria

\section{Introduction}

\setcounter{figure}{0}

An $m\times p$ matrix $M$ with entries from $\bbR$ is said to be \emph{totally
nonnegative} (tnn for short) if each of its minors is nonnegative. Further, such a matrix is
\emph{totally positive} if each of its minors is strictly positive. (Warning:
in some texts, the terms totally positive and strictly totally positive are
used for our terms totally nonnegative and totally positive, respectively.) 

Totally nonnegative matrices arise in many areas of mathematics and there has
been considerable interest lately in the study of these matrices. For
background information and historical references, there are the newly
published books by Fallat and Johnson, \cite{fj} and Pinkus, \cite{pinkus} and
also two good survey articles \cite{ando} and \cite{fz}.

In general, the number of minors of an $m\times m$ matrix is 
$\sum_{k=1}^m\,{m\choose k}^2 
= {2m\choose m} -1
\approx \frac{4^m}{\sqrt{\pi m}}   
$, 
by using Stirling's approximation to the factorial. 
Thus, it is obviously impractical to check whether a matrix is totally
nonnegative or totally positive by naively checking all minors. 

In this paper, we are concerned with the {\em recognition problem}: can one
decide, by checking a small number of minors,
whether or not a real matrix is totally
nonnegative or totally positive? 
There is an efficient set of
minors to test for total positivity: Gasca and Pe\~na, \cite{gp1}, specify
$mp$ minors that can be used to check total positivity. In order to describe
their result, we need to introduce some terminology. If a minor of a matrix is
formed by using rows from a set $I$ and columns from a set $J$ then we denote
the minor by $[I|J]$, or $[I|J](M)$ if it is necessary to specify that we are
using the matrix $M$. A minor is said to be an {\em initial minor} if $I$ and
$J$ each consists of a sequence of consecutive integers and $1\in I\cup J$.
Each initial minor is specified by its bottom right entry and so there are
$mp$ initial minors in an $m\times p$ matrix. Gasca and Pe\~na's criterion is
that a real matrix is totally positive if each of its initial minors is
strictly positive. The number of minors used in Gasca and
Pe\~na's criterion, $mp$, is best possible. However, there are many other
possible choices for a set of $mp$ minors which will give an efficient check
for total positivity. Interested readers should consult \cite{fz} where the
notion of double wiring diagrams is discussed and it is shown that each double
wiring diagram gives rise to a set of $n^2$ minors that can be used to check
for total positivity of $n\times n$ real matrices, see \cite[Theorem 16]{fz}.

In contrast, as far as we know, there is no small set of minors that can be
used to check for total nonnegativity. The best result we are aware of is a
result by Gasca and Pe\~na, \cite{gp2}, which specifies a set of minors of
size approximately the square root of the total number of minors of the matrix
 that can be used to check total nonnegativity of an invertible matrix, see
\cite[Theorem 29]{fz} and the remark following that theorem. 

This discrepancy in the two tests may be explained by considering the 
notion of the cell decomposition of the space of totally nonnegative matrices that was 
studied  by Postnikov in \cite{post}. 

Let $\mc=\mmptnn$ denote the set of $m\times p$ matrices that are totally
nonnegative. Let $Z$ be a set of minors and denote by $S^{\circ}_Z$ the set of
matrices $M\in\mmptnn$ for which $[I|J](M)=0$ if and only if $[I|J]\in Z$;
this set is known as the {\em (totally nonnegative) cell corresponding to the
set $Z$}. For many choices of $Z$, the cell corresponding to $Z$ will be
empty: for example, a $2\times 2$ matrix has five minors, so there are $32$
choices of $Z$ but it is not too difficult to check that only $14$ of these
choices give rise to nonempty cells. For example, it is an easy exercise to
check that the totally nonnegative cell corresponding to the choice
$Z:=\{[1|1]\}$ is empty. Nevertheless, in general, there are still a large
number of non-empty cells: a weak, but easily established, lower bound is that
$\mmptnn$ has more than $m^p$ nonempty cells.

The nonempty cells provide a stratification of the set $\mc$. The cell
corresponding to $Z=\emptyset$ is the set of totally positive matrices.
Although one can show that the set of totally positive matrices is dense in
the set of totally nonnegative matrices, nevertheless, it is only one of many
nonempty cells. 

The main aim of this paper (achieved in Theorem~\ref{theorem-main}) is to
provide a set of $mp$ minors for deciding whether or not an arbitrary real
$m\times p$ matrix belongs to a specified nonempty cell. This result extends
the criterion of Gasca and Pe\~na for the totally positive cell to any
nonempty cell. In the special case of invertible square matrices, our result
also extends related work of Fomin and Zelevinsky. In \cite[Theorems 4.1 and
4.13]{fz2} Fomin and Zelevinsky give recognition criteria for those totally
nonnegative cells in $n\times n$ square matrices which do not contain the
determinant. Note however that there are many totally nonnegative cells in
$\mc^{\geq 0}_{n}(\bbR)$ that do contain the determinant. For
example, in the $3\times 3$ case, there are 230 nonempty totally nonnegative
cells, of which 194 contain the determinant, and in the $4\times 4$ case,
there are 6902 nonempty totally nonnegative cells, of which 6326 contain
the determinant.

The methods used to construct our test were originally developed by 
Cauchon, \cite{cauchon2}, while studying quantum matrices. In recent papers, 
a close connection has been shown to exist between the cell decomposition 
of totally nonnegative matrices and the so-called invariant prime 
spectrum of quantum matrices, see, for example, \cite{gll1, gll2, ll} and 
it has become apparent that methods used in the quantum world may be used 
to suggest methods and results in the classical world of total nonnegativity,
 see, for an example, \cite{gl}. 

The tools we use here depend on the notion of Cauchon diagrams and 
the deleting derivation algorithm that Cauchon developed in \cite{cauchon2} 
for quantum matrices. Here, we propose that the corresponding algorithm 
for totally nonnegative matrices shall be called the 
{\em Cauchon reduction algorithm}. 

In the next section, we describe Cauchon diagrams and the
Cauchon reduction algorithm. We also introduce results from \cite{gll1} which
show how $m\times p$ Cauchon diagrams can be used to parameterise the nonempty
totally nonnegative cells in $\mmptnn$. As an example, we consider the totally
positive cell. This example motivates the discussion of lacunary sequences for
Cauchon diagrams which are introduced in Section~\ref{section-lacunary}.
Finally, in Section~\ref{section-recognition}, we show how, for each Cauchon
diagram, lacunary sequences can be used to specify a set of $mp$ minors which
can be used to test for membership of the totally nonnegative cell
corresponding to the given Cauchon diagram.

\section{Cauchon diagrams and Cauchon reduction}

Set $\gc 1,m\dc:=\{1,2,\dots, m\}$ for any $m\geq 1$. 
  
\begin{definition}\cite{cauchon2} {\rm
An $m\times p$ \emph{Cauchon diagram} $C$ is simply an $m\times p$ grid
consisting of $mp$ squares in which certain squares are coloured black. We
require that the collection of black squares has the following property: if a
square is black, then either every square strictly to its left is black or
every square strictly above it is black.

We denote by $\mathcal{C}_{m,p}$ the set of $m\times p$ Cauchon diagrams.
}
\end{definition}

Note that we will often identify an $m \times p $ Cauchon diagram with the set
of coordinates of its black boxes. Indeed, if $C \in \mathcal{C}_{m,p}$ and
$(\ia) \in \gc 1, m\dc \times \gc 1, p\dc$, we will say that $(\ia) \in C$ if
the box in row $i$ and column $\alpha$ of $C$ is black. For instance, for the
Cauchon diagram $C$ of Figure~\ref{fig:CD}, we have $(2,4) \in C$,
whereas $(4,3) \notin C$. (Note that we are using the usual
matrix notation for the $(i,\alpha)$ position in a Cauchon diagram; so that the
$(1,1)$ box of a Cauchon diagram is at the top left.)

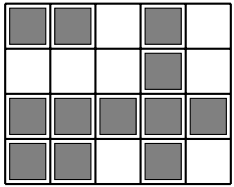
\begin{figure}[h]
\begin{center}
\begin{tikzpicture}[xscale=0.6,yscale=0.6]
\draw [fill=gray] (0.1,0.1) rectangle (0.9,0.9);
\draw [fill=gray] (1.1,0.1) rectangle (1.9,0.9);
\draw [fill=gray] (3.1,0.1) rectangle (3.9,0.9);
\draw [fill=gray] (0.1,1.1) rectangle (0.9,1.9);
\draw [fill=gray] (1.1,1.1) rectangle (1.9,1.9);
\draw [fill=gray] (2.1,1.1) rectangle (2.9,1.9);
\draw [fill=gray] (3.1,1.1) rectangle (3.9,1.9);
\draw [fill=gray] (4.1,1.1) rectangle (4.9,1.9);
\draw [fill=gray] (3.1,2.1) rectangle (3.9,2.9);
\draw [fill=gray] (0.1,3.1) rectangle (0.9,3.9);
\draw [fill=gray] (1.1,3.1) rectangle (1.9,3.9);
\draw [fill=gray] (3.1,3.1) rectangle (3.9,3.9);
\draw [help lines,black,thick] (0,0) grid (5,4);
\end{tikzpicture}
\caption{An example of a $4\times 5$ Cauchon diagram \label{fig:CD}}
\end{center}
\end{figure}

\begin{defi}
Let $X=(x_{i,\alpha})$ be a real $m\times p$ matrix and let 
$C$ be a Cauchon diagram (of size $m\times p$). 
We say that $X$ is a \emph{Cauchon
matrix associated to the Cauchon
diagram $C$} provided that for all 
$(\ia) \in \gc 1,m \dc \times \gc 1,p \dc$, we
have
$x_{\ia}=0$ if and only if $(\ia) \in C$. 
If $X$ is a Cauchon matrix associated to
an unnamed Cauchon diagram, we just say that 
$X$ is a \emph{Cauchon matrix}.
\end{defi}

A key link between Cauchon diagrams and totally nonnegative matrices 
is provided by the following easy lemma.

\begin{lemma} 
Every totally nonnegative matrix over $\bbR$ is a Cauchon matrix.
\end{lemma} 

\begin{proof} Let $X=(x_{i,\alpha})$ be a tnn matrix. Suppose that some $x_{\ia} =0$, and that $x_{k,\alpha} >0$ for some $k<i$. Let $\gamma<\alpha$. We need to prove that $x_{i,\gamma} =0$.
As $X$ is tnn, we have $-x_{k, \alpha} x_{i, \gamma} = \det 
\left(\begin{smallmatrix}
x_{k, \gamma} & x_{k, \alpha} \\ x_{i, \gamma} & x_{i, \alpha} 
\end{smallmatrix}\right)
\geq 0$. 
As $x_{k,\alpha} > 0$, this forces $x_{i,\gamma} \leq 0$. But since
$X$ is tnn, we also have $x_{i,\gamma} \geq 0$, so that $x_{i,\gamma} = 0$, as
desired. 
 
Therefore $X$ is a Cauchon matrix.
\end{proof}

In order to define the Cauchon reduction algorithm we need the following
notation.

\begin{nota}
\begin{itemize}
\item We denote by $\leq$ the lexicographic ordering on $\bbN^2$. Recall that
$$
(\ia) \leq (j,\beta) \Longleftrightarrow 
[(i < j) \mbox{ or } (i=j \mbox{ and } \alpha \leq \beta )].
$$
\item Set $\Ecirc= \left(\gc 1,m \dc \times \gc 1,p \dc \right)
\setminus \{(1,1)\}$ and $E= \Ecirc\cup \{(m+1,p)\}$.
\item Let $(j,\beta) \in \Ecirc$. 
Then $(j,\beta)^{+}$ denotes the smallest element 
(relative to $\leq$) of the set 
$\left\{ (\ia) \in E \ | \ (j,\beta) < (\ia) \right\}$.
\end{itemize}
 \end{nota}
Thus, $(j,\beta)^{+}=(j,\beta+1)$ when $\beta<p$ and $(j,p)^+=(j+1,p)$.

\begin{conv}[Cauchon reduction algorithm]
\label{alg1}
Let $M=(x_{\ia})$ be a real $m\times p$  matrix. 
As $r$ runs over the set $E$, we define matrices 
$M^{(r)} :=(x_{\ia}^{(r)})$ as follows.
\begin{enumerate}
\item \underline{When $r=(m+1,p)$}, the entries of the matrix 
$M^{(m+1,p)}$ are defined by 
$x_{\ia}^{(m+1,p)}:=x_{\ia}$ for all 
$(\ia) \in \gc 1,m \dc \times \gc 1,p \dc$; so $M^{(m+1,p)}=M$. 
\item \underline{Assume that $r=(j,\beta) \in \Ecirc$} and that 
the matrix $M^{(r^{+})}=(x_{\ia}^{(r^{+})})$ has 
already been defined. The entries $x_{\ia}^{(r)}$ of the matrix 
$M^{(r)}$ are defined as follows.
\begin{enumerate}
\item If $x_{j,\beta}^{(r^+)}=0$, then $x_{\ia}^{(r)}=x_{\ia}^{(r^+)}$ 
for all $(\ia) \in \gc 1,m \dc \times \gc 1,p \dc$.
\item If $x_{j,\beta}^{(r^+)}\neq 0$ and 
$(\ia) \in \gc 1,m \dc \times \gc 1,p \dc$, then 

$x_{\ia}^{(r)}= \left\{ \begin{array}{ll}
x_{\ia}^{(r^+)}-x_{i,\beta}^{(r^+)} \left( x_{j,\beta}^{(r^+)}\right)^{-1} 
x_{j,\alpha}^{(r^+)}
& \qquad \mbox{if }i <j \mbox{ and } \alpha < \beta \\
x_{\ia}^{(r^+)} & \qquad \mbox{otherwise.} \end{array} \right.$ 
\end{enumerate}
We say that {\em $M^{(r)}$ is the matrix obtained from $M$ by 
applying the Cauchon reduction  
algorithm at step $r$}, and $x_{j,\beta}^{(r^+)}$ is called the 
{\em pivot at step $r$}.
\item \underline{If $r=(1,2)$}, then we set $t_{\ia}:=x_{\ia}^{(1,2)}$ 
for all $(\ia) \in \gc 1,m \dc \times \gc 1,p \dc$. Observe that 
$x^{(r)}_{\ia}= x^{(r^+)}_{\ia}$ for all $r\le (\ia)$, 
and so $t_{\ia}= x^{(\ia)}_{\ia}= x^{(\ia)^+}_{\ia}$ 
for all $(\ia)\in \Ecirc$.  The matrix $\widetilde{M}:=M^{(1,2)}$ 
is {\em the matrix obtained from $M$ at the end of the Cauchon reduction 
algorithm}.
\end{enumerate}
\end{conv}

In \cite{gll1} it is proved that the Cauchon reduction algorithm provides a bijection between the
nonempty cells of $\mmptnn$ and the Cauchon diagrams of size $m\times p$. Let us give more details about this bijection. 
First, let us recall the following result from \cite{gll1}.
\begin{theorem}
\label{thm:cell1}
A real $m\times p$ matrix $M$ is 
totally nonnegative if and only if $\widetilde{M}$, the matrix obtained 
from $M$ at the end of the Cauchon reduction algorithm, is a nonnegative 
Cauchon matrix.
\end{theorem}
\begin{proof}
This is the content of \cite[Theorem 4.1 and Corollary B5]{gll1}.
\end{proof}  

Thus, each totally nonnegative matrix is, via Cauchon 
reduction, associated to a Cauchon diagram. In other words, we have a mapping $\pi: M \mapsto C$ from $\mmptnn$ to the set  $\mathcal{C}_{m,p}$ of $m\times p$ Cauchon diagrams, where $C$ is the Cauchon diagram associated to the Cauchon matrix $\widetilde{M}$. 
Moreover, if $M$ and $N$ are two tnn matrices such that $\pi (M)=\pi (N)$, then it follows from \cite[Corollary 3.17]{gll1} that $M$ and $N$ belong to the same tnn cell. So, each nonempty tnn cell in $\mmptnn$ is a union of fibres of $\pi$. On the other hand, it follows from \cite{post} (see also \cite[Section 6]{gll1}) that the number of nonempty tnn cells in $\mmptnn$ is equal to the number of $m\times p$ Cauchon diagrams. Thus, the nonempty tnn cells are precisely the fibres of $\pi$, and we have just proved the following result.

\begin{theorem}
\label{thm:cell2}
\begin{enumerate}
\item Let $M$ and $N$ be two tnn matrices. Then $M$ and $N$ belong to the same tnn cell if and only if $\pi (M)=\pi (N)$, ie if and only if $\widetilde{M}$ and $\widetilde{N}$ are associated to the same Cauchon diagram.
\item Nonempty tnn cells are parametrised by Cauchon diagrams, and the nonempty tnn cells in  $\mmptnn$ are precisely the sets 
$$S^0_C:=\{ M \in \mmptnn ~|~ \pi(M)=C\},$$
where $C$ runs through the set of $m \times p$ Cauchon diagrams. 
\item Let $M \in \mmptnn$, and denote by $\widetilde{M}=(t_{i,\alpha})$ the matrix obtained 
from $M$ at the end of the Cauchon reduction algorithm. Then $M \in S^0_C$ if and only if $t_{i,\alpha} = 0$ if $(i,\alpha) \in C$ and  $t_{i,\alpha} > 0$ if $(i,\alpha) \notin C$.
\end{enumerate}
\end{theorem}
\begin{proof}
The first two statements follow from the above discussion. The third assertion follows from the fact that $\widetilde{M}$ is a nonnegative Cauchon matrix associated to $C$ if and only if $t_{i,\alpha} = 0$ if $(i,\alpha) \in C$ and  $t_{i,\alpha} > 0$ if $(i,\alpha) \notin C$.
\end{proof}

Let $C$ be a Cauchon diagram. The tnn cell $S^0_C$ is called the {\it tnn cell
associated to the Cauchon diagram $C$}.

As an example, we will consider the totally positive cell.
In order to reconcile the Gasca and Pe\~na approach with our approach, we need
to describe a symmetry of totally nonnegative matrices. The reason is that we
are using Cauchon reduction, which starts at the bottom right of a matrix and
moves along columns from right to left, whereas Gasca and Pe\~na use
techniques such as Neville elimination, which start at the top left and move
down columns. This makes no essential difference, except to the notation:
informally, reflection in the antidiagonal preserves total nonnegativity. 

 To be precise, let
$M=(a_{ij})$ be an $m\times p$ matrix and let $M^{\rho}$ be the $p\times m$
matrix defined by $(M^{\rho})_{ij}:= a_{m+1-j,\,p+1-i}$. Let $I$ and $J$ be
row and columns sets for $M^{\rho}$. Then standard properties of determinants
show that $[I|J](M^{\rho})=[m+1-J|p+1-I](M)$. It follows that $M$ is totally
nonnegative if and only if $M^{\rho}$ is totally nonnegative, and, similarly,
$M$ is totally positive if and only if $M^{\rho}$ is totally positive. (Here,
$x+I:=\{x+i~|~i\in I\}$.)

We say that a minor $[I|J]$ is a {\em final minor} of an $m\times p$ matrix if
each of $I$ and $J$ consists of a consecutive set of integers and either $m\in
I$ or $p\in J$. It is clear that initial minors of $M$
correspond to final minors of $M^{\rho}$; so the Gasca and Pe\~na criterion
can be reformulated as: an $m\times p$ matrix $M$ is totally positive if and
only if each of its $mp$ final minors is strictly positive.

If the Cauchon reduction algorithm is carried out on a totally positive 
matrix $M$ then one can calculate that $t_{i,\alpha}$, the $(i,\alpha)$ 
entry of 
$\widetilde{M}$, the matrix obtained at the end of the Cauchon reduction
algorithm, is given by
\[
t_{i,\alpha}=
[i,i+1,\dots,i+r~|~\alpha,\alpha+1,
\dots,\alpha+r](M) \cdot  [i+1,\dots,i+r~|~\alpha+1,
\dots,\alpha+r](M)^{-1},
\]
 where $r=\min\{m-i, p-\alpha\}$. Note that 
this formula only involves final minors of $M$. In particular, 
$\widetilde{M}$ is a strictly positive matrix and so the Cauchon diagram
corresponding to the totally positive cell has all boxes coloured 
white.  
Also, one can then easily calculate that for any final minor 
$[i,i+1,\dots,i+r~|~\alpha,\alpha+,1 
\dots,\alpha+r~](M)=t_{i,\alpha}t_{i+1,\alpha+1}\cdots t_{i+r,\alpha+r}$. 

In considering totally nonnegative cells other than the totally positive cell,
two problems arise: some of the $t_{i,\alpha}$ will be equal to zero, and some
of the (final) minors of the matrix $M$ will be equal to zero, and so cannot
be inverted. In seeking a criterion for recognition of these cells that is
similar to the Gasca and Pe\~na criterion for total positivity we were led to
consider certain sequences called {\em lacunary sequences} which were first
introduced by Cauchon in his study of quantum matrices, \cite{cauchon2}. These
sequences are described in the next section.

\section{Lacunary sequences}\label{section-lacunary}

\begin{mydef}
\label{def:LS}
Let $C$ be an $m \times p$ Cauchon diagram. We say that a sequence 
\[
((i_0,\alpha_0), (i_1,\alpha_1), ..., (i_t,\alpha_t))
\]
is a {\em lacunary sequence with respect to $C$} if the 
following conditions hold:
\begin{enumerate}
\item $t \geq 0$;
\item the boxes $(i_1,\alpha_1)$, $(i_2,\alpha_2)$, ..., $(i_t,\alpha_t)$ are white in $C$; 
\item $1 \leq i_0 < i_1 < \dots < i_t  \leq m$ and $1 \leq \alpha_0 < \alpha_1 < \dots < \alpha_t  \leq p$;
\item If $i_t < i \leq m$ and $\alpha_t < \alpha \leq p$, then $(i,\alpha)$ is a black box in $C$;
\item Let $s \in \{0, \dots ,t-1\}$. Then:
\begin{itemize}
\item either $(i,\alpha)$ is a black box in $C$ for all $i_s < i <i_{s+1}$ 
and $\alpha_s <\alpha$,
\item  or $(i,\alpha)$ is a black box in $C$ for all 
$i_s < i <i_{s+1}$ and $\alpha_0\leq\alpha < \alpha_{s+1}$; 
\end{itemize}
(See Figure \ref{fig:Condition5} where $x_s^+:=x_{s+1}$ and $x_k:=(i_k,\alpha_k)$.)
\item Let $s \in \{0, \dots ,t-1\}$. Then:
\begin{itemize}
\item either $(i,\alpha)$ is a black box in $C$ for all $i_s <i$ and $\alpha_s < \alpha <\alpha_{s+1}$,
\item  or $(i,\alpha)$ is a black box in $C$ for all $i <i_{s+1}$ and $\alpha_s < \alpha <\alpha_{s+1}$.
\end{itemize}
(See Figure \ref{fig:Condition6}.)
\end{enumerate}
\end{mydef}

\begin{example}
Consider the $3\times 3$ Cauchon diagram 
\begin{center}
\begin{tikzpicture}[xscale=0.6, yscale=0.6]
\node at (0,1.5) {$C=$};
\draw [fill=gray] (2.1,1.1) rectangle (2.9,1.9);
\draw [fill=gray] (2.1,0.1) rectangle (2.9,0.9);
\draw [fill=gray] (4.1,1.1) rectangle (4.9,1.9);
\draw [fill=gray] (4.1,2.1) rectangle (4.9,2.9);
\draw [help lines, black, thick] (2,0) grid   (5,3);
\end{tikzpicture}
\end{center}
It is easy to check that the sequence $((1,2), (3,3))$ is a lacunary sequence (with respect to $C$).
\end{example}


\begin{figure}[h]
\begin{center}
\begin{tikzpicture}[xscale=0.55,yscale=0.55]


\draw [ultra thick,red] (1,8) -- (1,9) -- (2,9) -- (2,8) -- (1,8);
\node [red] at (1.5,8.5) {{\bf \mbox{$x_0$}}};

\node [ultra thick,red] at (2.5,7.5) {{\bf \mbox{$\ddots$}}};

\draw [ultra thick,red] (3,6) -- (3,7) -- (4,7) -- (4,6) -- (3,6);
\node [red] at (3.5,6.5) {{\bf \mbox{$x_s$}}};

\draw [ultra thick,red] (7,1) -- (7,2) -- (8,2) -- (8,1) -- (7,1);
\node [red] at (7.6,1.5) {{\bf \mbox{$x_{s}^+$}}};

\draw [fill=gray] (4,2) rectangle (10,6);

\draw[ultra thin]   (0,0) grid   (10,10);

\node at (12,5) {{\bf OR }};

\draw [ultra thick,red] (15,8) -- (15,9) -- (16,9) -- (16,8) -- (15,8);
\node [red] at (15.5,8.5) {{\bf \mbox{$x_0$}}};

\node [ultra thick,red] at (16.5,7.5) {{\bf \mbox{$\ddots$}}};

\draw [ultra thick,red] (17,6) -- (17,7) -- (18,7) -- (18,6) -- (17,6);
\node [red] at (17.5,6.5) {{\bf \mbox{$x_s$}}};

\draw [ultra thick,red] (21,1) -- (21,2) -- (22,2) -- (22,1) -- (21,1);
\node [red] at (21.6,1.5) {{\bf \mbox{$x_{s}^+$}}};

\draw [fill=gray] (15,2) rectangle (21,6);

\draw[ultra thin]  (14,0) grid   (24,10);


\end{tikzpicture}
\caption{Condition 5 of Definition \ref{def:LS} \label{fig:Condition5}}
\end{center}
\end{figure}
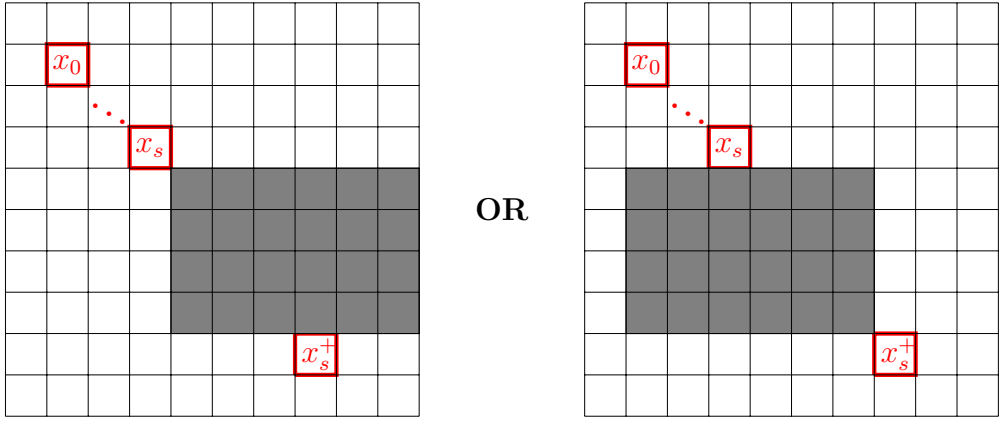

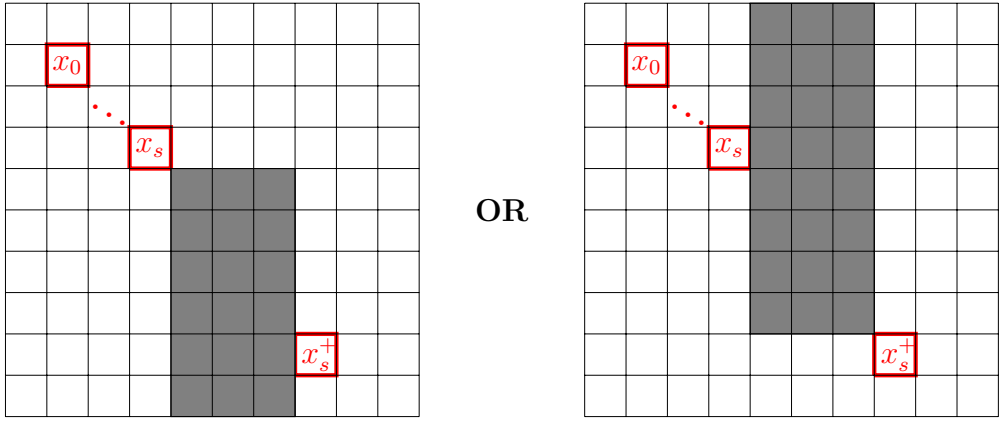
\begin{figure}[h]
\begin{center}
\begin{tikzpicture}[xscale=0.55,yscale=0.55]


\draw [ultra thick,red] (1,8) -- (1,9) -- (2,9) -- (2,8) -- (1,8);
\node [red] at (1.5,8.5) {{\bf \mbox{$x_0$}}};

\node [ultra thick,red] at (2.5,7.5) {{\bf \mbox{$\ddots$}}};

\draw [ultra thick,red] (3,6) -- (3,7) -- (4,7) -- (4,6) -- (3,6);
\node [red] at (3.5,6.5) {{\bf \mbox{$x_s$}}};

\draw [ultra thick,red] (7,1) -- (7,2) -- (8,2) -- (8,1) -- (7,1);
\node [red] at (7.6,1.5) {{\bf \mbox{$x_{s}^+$}}};

\draw [fill=gray] (4,0) rectangle (7,6);

\draw[ultra thin]   (0,0) grid   (10,10);

\node at (12,5) {{\bf OR }};

\draw [ultra thick,red] (15,8) -- (15,9) -- (16,9) -- (16,8) -- (15,8);
\node [red] at (15.5,8.5) {{\bf \mbox{$x_0$}}};

\node [ultra thick,red] at (16.5,7.5) {{\bf \mbox{$\ddots$}}};

\draw [ultra thick,red] (17,6) -- (17,7) -- (18,7) -- (18,6) -- (17,6);
\node [red] at (17.5,6.5) {{\bf \mbox{$x_s$}}};

\draw [ultra thick,red] (21,1) -- (21,2) -- (22,2) -- (22,1) -- (21,1);
\node [red] at (21.6,1.5) {{\bf \mbox{$x_{s}^+$}}};

\draw [fill=gray] (18,2) rectangle (21,10);

\draw[ultra thin]   (14,0) grid   (24,10);


\end{tikzpicture}
\caption{Condition 6 of Definition \ref{def:LS} \label{fig:Condition6}}
\end{center}
\end{figure}

In \cite{cauchon2}, lacunary sequences were studied in order to prove the normal
separation of the prime spectrum of the algebra of $m \times p$ quantum
matrices. Moreover, Cauchon studied the existence of such sequences. In
particular, he proved the following result. 

\begin{lemma}
\label{lemmaCauchon}
Let $(j,\beta) \in \gc 1, m \dc \times \gc 1,p \dc$. Assume there exists $j_0>j$ and $\beta_0 > \beta$ such that $(j_0,\beta_0)$ is a white box in $C$. Then there exists  $(l,\gamma) \in \gc 1, m \dc \times \gc 1, p \dc$ such that 
\begin{enumerate}
\item $l>j $ and $ \gamma > \beta$;
\item the box $(l,\gamma)$ is white in $C$; 
\item
\begin{itemize}
\item Either $(i,\alpha)$ is a black box in $C$ for all $j < i <l$ and $\beta <\alpha$,
\item  or $(i,\alpha)$ is a black box in $C$ for all $j < i <l$ and $\alpha < \gamma$; 
\end{itemize}
\item 
\begin{itemize}
\item Either $(i,\alpha)$ is a black box in $C$ for all $j <i$ and $\beta < \alpha <\gamma$,
\item  or $(i,\alpha)$ is a black box in $C$ for all $i <l$ and $\beta < \alpha <\gamma$.
\end{itemize}

\end{enumerate}
\end{lemma}
\begin{proof}
We give a proof of this lemma for the convenience of the reader. We distinguish between three cases.

{\bf Case 1:} The first $\beta$ columns of the Cauchon diagram $C^{r}$ obtained from $C$ by deleting the first $j$ rows are black. By assumption, there is a white box in $C^r$. Let $\gamma$ be the smallest column-index of a non-black column in $C^r$. Hence $\gamma > \beta$. We know that there is a white box in the column $\gamma$ of $C^r$, so 
$$l:= \min \{k ~|~j<k \mbox{ and the box $(k,\gamma)$ is white}\} $$
exists and $j<l \leq m$.

\begin{figure}[hp]
   \begin{minipage}[c]{.46\linewidth}
   
\parbox{5ex}{~~~~~}
\parbox{35ex}{
\begin{tikzpicture}[xscale=0.5,yscale=0.5]

\draw[help lines]   (2,2) grid   (12,12);
\node at (5.5, 12.5) {$\beta$};
\node at (1.5, 9.5) {$j$};

\node [blue] at (8.5, 12.5) {$\gamma$};
\node [blue] at (1.5, 5.5) {$l$};

\draw (1,10) -- (12,10);
\draw (1,9) -- (12,9);

\draw (5,2) -- (5,13);
\draw (6,2) -- (6,13);

\draw [fill=darkgray] (2,2) rectangle (6,9);
\draw [fill=gray] (6,2) rectangle (8,9);
\draw [fill=gray] (8,6) rectangle (9,9);

\draw [ultra thick,red] (2,2) rectangle (12,9);

\node [blue] at (8.5,5.5) {{\bf W}};
\draw [blue]  (8,13) -- (8,5);
\draw [blue]  (9,13) -- (9,5);

\draw [blue] (1,5) -- (9,5);
\draw [blue] (1,6) -- (9,6);

\end{tikzpicture}
}
   \end{minipage} \hfill
   \begin{minipage}[c]{.46\linewidth}
   \parbox{35ex}{ In this picture, the area in red corresponds to the Cauchon
   diagram $C^r$, and the box marked with a W is white.}\parbox{5ex}{~~~~~}
   \end{minipage}
   \caption{Proof of Lemma \ref{lemmaCauchon}, Case 1.}
   \label{figure:case1}
\end{figure}

We claim that $(l,\gamma)$ satisfies the properties of the lemma. Indeed, if
$j < i <l$ and $\alpha < \gamma$, then the box $(i,\alpha)$ is black as by
construction the column indexed $\alpha$ of $C^r$ is black. So the second
assertion of 3. is satisfied. 

Similarly, if $j <i$ and $\beta < \alpha <\gamma$, then the box $(i,\alpha)$
is black, so that the first assertion of 4. is satisfied. \\


{\bf Case 2:} The first $j$ rows of the Cauchon diagram $C^{c}$ obtained from
$C$ by deleting the first $\beta$ columns are black. By assumption, there is a
white box in $C^c$. Let $l$ be the smallest row-index of a non-black row in
$C^c$. Hence $l > j$. We know that there is a white box in the row $l$ of
$C^c$, so $$\gamma:= \min \{\eta ~|~\beta<\eta \mbox{ and the box $(l,\eta)$
is white}\} $$ exists and $\beta<\gamma \leq p$.


\begin{figure}[hp]
   \begin{minipage}[c]{.46\linewidth}
   
\parbox{5ex}{~~~~~}
\parbox{35ex}{

\begin{tikzpicture}[xscale=0.5,yscale=0.5]

\draw[help lines]   (2,2) grid   (12,12);
\node at (5.5, 12.5) {$\beta$};
\node at (1.5, 9.5) {$j$};

\node [blue] at (8.5, 12.5) {$\gamma$};
\node [blue] at (1.5, 5.5) {$l$};

\draw (1,10) -- (12,10);
\draw (1,9) -- (12,9);

\draw (5,2) -- (5,13);
\draw (6,2) -- (6,13);

\draw [fill=darkgray] (6,9) rectangle (12,12);
\draw [fill=gray] (6,6) rectangle (12,9);
\draw [fill=gray] (6,5) rectangle (8,6);

\draw [ultra thick,red] (6,2) rectangle (12,12);

\node [blue] at (8.5,5.5) {{\bf W}};
\draw [blue]  (8,13) -- (8,5);
\draw [blue]  (9,13) -- (9,5);

\draw [blue] (1,5) -- (9,5);
\draw [blue] (1,6) -- (9,6);

\end{tikzpicture}
}

   \end{minipage} \hfill
   \begin{minipage}[c]{.46\linewidth}
   \parbox{35ex}{ In this picture, the area in red corresponds to the Cauchon
   diagram $C^c$, and the box marked with a W is white.}\parbox{5ex}{~~~~~}
   \end{minipage}
   \caption{Proof of Lemma \ref{lemmaCauchon}, Case 2.}
   \label{figure:case2}
\end{figure}


We claim that $(l,\gamma)$ satisfies the properties of the lemma. Indeed, if  $j < i <l$ and $\beta < \alpha$, then the box $(i,\alpha)$ is black as by construction the row indexed $i$ of $C^c$ is black. So the first assertion of 3. is satisfied. 

Similarly, if $i<l$ and $\beta < \alpha <\gamma$, then the box $(i,\alpha)$ is
black, so that the second assertion of 4. is satisfied. \\


{\bf Case 3:} The first $\beta$ columns of the Cauchon diagram $C^{r}$ obtained from $C$ by deleting the first $j$ rows are not entirely black, and the first $j$ rows of the Cauchon diagram $C^{c}$ obtained from $C$ by deleting the first $\beta$ columns are not entirely black. 

We set 
$$k:= \min \{s ~|~j<s \mbox{ and there exists $1\leq \eta \leq \beta$ such that the box $(s,\eta)$ is white}\} $$
and 
$$\gamma := \min \{\eta ~|~\beta< \eta \mbox{ and there exists $1\leq s \leq j$ such that the box $(s,\eta)$ is white}\} .$$
We have $j < k \leq m$ and $\beta < \gamma \leq p$. 

Observe that the box $(k,\gamma)$ is also white because $C$ is a Cauchon diagram. Now we can set 
$$l:=  \min \{s ~|~j<s \mbox{ and the box $(s,\gamma)$ is white}\} .$$


\begin{figure}[hp]
   \begin{minipage}[c]{.46\linewidth}
   
\parbox{5ex}{~~~~~}
\parbox{35ex}{
\begin{tikzpicture}[xscale=0.55,yscale=0.55]

\draw[help lines]   (2,2) grid   (12,12);
\node at (5.5, 12.5) {$\beta$};
\node at (1.5, 9.5) {$j$};

\node [blue] at (8.5, 12.5) {$\gamma$};
\node [blue] at (1.5, 5.5) {$l$};
\node [blue] at (1.5, 3.5) {$k$};

\draw (1,10) -- (12,10);
\draw (1,9) -- (12,9);

\draw (5,2) -- (5,13);
\draw (6,2) -- (6,13);

\draw [fill=darkgray] (6,9) rectangle (8,12);
\draw [fill=darkgray] (8,11) rectangle (9,12);
\draw [fill=darkgray] (2,3) rectangle (4,4);

\draw [fill=gray] (8,6) rectangle (9,9);
\draw [fill=lightgray] (6,6) rectangle (8,9);

\draw [fill=darkgray] (2,4) rectangle (6,9);

\draw [ultra thick,red] (6,9) rectangle (12,12);
\draw [ultra thick,red] (2,2) rectangle (6,9);

\node [blue] at (8.5,10.5) {{\bf $W_0$}};
\node [blue] at (8.5,3.5) {{\bf $W_1$}};
\node [blue] at (8.5,5.5) {{\bf $W_2$}};
\node [blue] at (4.5,3.5) {{\bf $W_0$}};
\draw [blue]  (8,13) -- (8,3);
\draw [blue]  (9,13) -- (9,3);

\draw [blue] (1,5) -- (9,5);
\draw [blue] (1,6) -- (9,6);

\draw [blue] (1,3) -- (9,3);
\draw [blue] (1,4) -- (9,4);

\end{tikzpicture}
}
   \end{minipage} \hfill
   \begin{minipage}[c]{.46\linewidth}
   \parbox{35ex}{
    In this picture, the areas in red correspond to the two areas containing a
    white box, and the boxes marked with W are white.}\parbox{5ex}{~~~~~}
   \end{minipage}
   \caption{Proof of Lemma~\ref{lemmaCauchon}, Case 3.}
   \label{figure:case3}
\end{figure}

We claim that $(l,\gamma)$ satisfies the properties of the lemma. Indeed,
if $j < i <l$ and $\alpha < \gamma$, then the box $(i,\alpha)$ is black.
Otherwise, $(i,\alpha)$ is white and by construction of $\gamma$ there exists
$1\leq s \leq j$ such that the box $(s,\gamma)$ is white. As $C$ is a Cauchon
diagram, this would force $(i,\gamma)$ to be white, contradicting the
minimality of $l$. Hence, $(i,\alpha)$ is black, and the second assertion of
3. is satisfied. 

Now, assume that $i<l$ and $\beta < \alpha <\gamma$. We distinguish between
two cases to show that the box $(i,\alpha)$ is black. If $j<i$, then the same
argument as above shows that $(i,\alpha)$ is black. Indeed, otherwise
$(i,\alpha)$ is white and by construction of $\gamma$ there exists $1\leq s
\leq j$ such that the box $(s,\gamma)$ is white. As $C$ is a Cauchon diagram,
this would force $(i,\gamma)$ to be white, contradicting the minimality of
$l$. Hence, $(i,\alpha)$ is black. Finally, if $i\leq j$, then by construction
of $\gamma$, the box $(i,\alpha)$ is black. Hence, the second assertion of 4.
is satisfied. 
\end{proof}

Fix $(j,\beta) \in \gc 1, m \dc \times \gc 1, p \dc$. The proof of the previous lemma
actually gives us an algorithmic way to produce a lacunary sequence starting
at $(j,\beta)$. 


\begin{algorithm}
\caption{Algorithm constructing a lacunary sequence with respect to $C$ from any box}
\label{algo1}

\begin{algorithmic}
\State {\bf input:} Let $C$ be a Cauchon diagram, and fix $(j,\beta) \in \gc 1, m \dc \times \gc 1, p \dc$.
\State {\bf output:} a lacunary sequence $((i_0,\alpha_0), (i_1,\alpha_1), ..., (i_t,\alpha_t))$ with respect to $C$ such that $(i_0,\alpha_0)=(j,\beta).$
\State {\bf start:} $i_0:=j$, $\alpha_0:=\beta$, $t:=0$, $i:=i_t$, $\alpha:=\alpha_t$ 
\While{ the Cauchon diagram $C_{i,\alpha}$ deduced from $C$ by deleting the first $i$ rows and the first $\alpha$ columns 
is not completely black}
\If{  the first $\alpha$ columns of the Cauchon diagram $C^{i}$ obtained from $C$ by deleting the first $i$ rows are black} 
\State {$\alpha_{t+1}:=$the smallest column-index of a non-black column in $C^{i}$ }
\State { {\bf and} $i_{t+1}:= \min \{k ~|~i<k \mbox{ and the box $(k,\alpha_{t+1})$ is white}\} $}
\Else \If {the first $i$ rows of the Cauchon diagram $C^{\alpha}$ obtained from $C$ by deleting the first $\alpha$ columns are black}
\State{ $i_{t+1}:=$ the smallest row-index of a non-black row in $C^{\alpha}$}
\State { {\bf and} $\alpha_{t+1}:= \min \{\eta ~|~\alpha<\eta \mbox{ and the box $(i_{t+1},\eta)$ is white}\} $}
\Else  \State {$\alpha_{t+1} := \min \{\eta ~|~\alpha< \eta \mbox{ and there exists $1\leq s \leq i$ such that the box $(s,\eta)$ is white}\}$}
\State {{\bf and} $i_{t+1}:=  \min \{s ~|~i<s \mbox{ and the box $(s,\alpha_{t+1})$ is white}\} $}
 \EndIf  
 \EndIf  
 \State { $t:=t+1$.}
 \EndWhile  
    
\noindent \Return   $((i_0,\alpha_0),(i_1,\alpha_1), ..., (i_t,\alpha_t))$.
\end{algorithmic}
\end{algorithm}

\newpage

\begin{corollary}
Let $(j,\beta) \in \gc 1, m \dc \times \gc 1, p \dc$. Then there exists a lacunary sequence 
$$((i_0,\alpha_0),(i_1,\alpha_1), ..., (i_t,\alpha_t))$$ with respect to $C$ with $$(i_0,\alpha_0)=(j,\beta).$$
\end{corollary}
\begin{proof}
Algorithm \ref{algo1} produces such a lacunary sequence thanks to the previous lemma.
\end{proof}

\begin{example}
Consider the Cauchon diagram
\begin{center}
\begin{tikzpicture}[xscale=0.6, yscale=0.6]
\node at (0,1.5) {$C=$};
\draw [fill=gray] (2.1,1.1) rectangle (2.9,1.9);
\draw [fill=gray] (3.1,1.1) rectangle (3.9,1.9);
\draw [fill=gray] (4.1,2.1) rectangle (4.9,2.9);
\draw [help lines, black, thick] (2,0) grid   (5,3);
\end{tikzpicture}
\end{center}
Applying the above algorithm starting from the box $(1,2)$, we obtain the lacunary sequence $((1,2),(2,3))$. 
Note that this is not the only lacunary sequence starting at $(1,2)$. Indeed, one can easily check that 
$((1,2),(3,3))$ is another lacunary sequence starting at $(1,2)$.
\end{example}

\begin{remark}
\label{rk:LS}
 If $j=m$ or $\beta=p$, then the only lacunary sequence (with respect to any $m\times p$ Cauchon diagram) 
is $((j,\beta ))$. 
\end{remark}

\section{Recognition}\label{section-recognition}

Fix a Cauchon diagram $C$. If $(j,\beta) \in \gc 1, m \dc \times \gc 1, p \dc$, then we
denote by $ \mathrm{LC}_{j,\beta}$ the set of all lacunary sequences (for $C$)
that starts at $(j,\beta)$. The previous section shows that $
\mathrm{LC}_{j,\beta} \neq \emptyset$. If
$(\underline{j},\underline{\beta}):=((j_0,\beta_0), (j_1,\beta_1), ...,
(j_t,\beta_t)) \in \mathrm{LC}_{j,\beta}$ is a lacunary sequence starting at
$(j,\beta)$, then we set: $$\Delta_{\underline{j},\underline{\beta}}^C:=
[j_0,j_1,\dots ,j_t ~|~\beta_0 , \beta_1 , \dots , \beta_t] \in \pmmpc .$$

\begin{proposition}
\label{proposition:a-or-b}
Let $M \in \mc_{m,p}(\bbR)$. For all $(j,\beta) \in \gc 1,m\dc \times \gc 1,p \dc$, fix a
lacunary sequence $(\underline{j},\underline{\beta})=((j_0,\beta_0),
(j_1,\beta_1), ..., (j_d,\beta_d)) \in \mathrm{LC}_{j,\beta}$, and set
$\Delta_{j,\beta}^C := \Delta_{\underline{j},\underline{\beta}}^C$.
Assume that either (a) or (b) below is true.\\

(a) The matrix $M$ is tnn and belongs to the tnn cell $S^0_C$ associated to $C$.\\

(b) For all $(j,\beta)$, we have $\Delta_{j,\beta}^C(M) =0$ if
$(j,\beta) \in C$ and $\Delta_{j,\beta}^C(M) > 0$ if $(j,\beta) \notin C$.
\\

Then, 
$$\Delta_{\underline{i},\underline{\alpha}}^C(M)=t_{i_0,\alpha_0} \cdot
t_{i_1,\alpha_1} \cdot\;\cdots \;\cdot t_{i_t,\alpha_t}$$ for all 
lacunary sequences $(\underline{i},\underline{\alpha})=((i_0,\alpha_0), (i_1,\alpha_1), ..., (i_t,\alpha_t))$.
\end{proposition}

\begin{proof}
We prove by a decreasing induction on
$(j,\beta)\in \gc 1,m \dc \times \gc 1,p \dc$ (with respect to the
lexicographic order) that 

    \begin{itemize}

\item[(i)] For all $(l,\delta) \geq (j,\beta)$, then $t_{l,\delta} \geq 0$ and
$t_{l,\delta} = 0$ if and only if $(l,\delta)\in C$;
\item[(ii)] For all lacunary sequences
$(\underline{i},\underline{\alpha})=((i_0,\alpha_0), (i_1,\alpha_1), ...,
(i_t,\alpha_t)) \in \mathrm{LC}_{j,\beta}$, and for all $(k,\eta)\in E^0$
with $(k,\eta) \geq (j,\beta)$, we have
\[
\Delta_{\underline{i},\underline{\alpha}}^C(M)= [i_0,\dots, i_s ~|~\alpha_0,
 \dots, \alpha_s](M^{(k,\eta)}) t_{i_{s+1},\alpha_{s+1}} \cdots
t_{i_t,\alpha_t},
\]
where $(i_s,\alpha_s) < (k,\eta) \leq (i_{s+1},\alpha_{s+1})$. 

(By convention, if $s=-1$ (ie if $(k,\eta)=(j,\beta)=(i_0,\alpha_0)$), then we
set $$[i_0,\dots, i_s ~|~\alpha_0, \dots, \alpha_s](M^{(k,\eta)}) = 1.$$ On
the other hand, if $s=t$, then we set $(i_{t+1},\alpha_{t+1})=(m+1,p+1)$.)
    \end{itemize}

Note that condition (i) is automatically satisfied under assumption (a) 
thanks to Theorem \ref{thm:cell2}.

As observed in Remark \ref{rk:LS}, the only lacunary sequence starting at
$(m,\beta)$ is $((m,\beta))$. So it is clear that
$\Delta_{\underline{i},\underline{\alpha}}^C(M)=\Delta_{m,\beta}^C(M)
=x_{m,\beta}=t_{m,\beta}$ for all lacunary sequences starting at $(m,\beta)$.
This proves (ii) in the case where $j=m$. Moreover, under assumption (b), as
$\Delta_{j,\beta}^C(M) =0$ if $(j,\beta) \in C$ and $\Delta_{j,\beta}^C(M) >
0$ if $(j,\beta) \notin C$, we obtain that $t_{m,\beta} =0$ if $(m,\beta) \in
C$ and $t_{m,\beta} > 0$ if $(m,\beta) \notin C$, which proves (i) in the case
where $j=m$. 

Let $j<m$, and assume (i) and (ii) are true for all $(j',\beta')>(j,\beta)$.
In particular, we assume that $t_{l,\delta} \geq 0$ for all
$(l,\delta)>(j,\beta)$ with equality exactly when $(l,\delta)\in C$. 

To prove (ii) at step $(j,\beta)$, we proceed by a decreasing secondary induction on
$(k,\eta)$. 

When $(k,\eta)=(m+1,p)$, the result is trivial. So let 
$ (j,\beta)\leq \;(k,\eta) < (m+1,p)$
and assume the result established when $(k,\eta)$ is replaced by $(k,\eta)^+$.
In particular, for the lacunary sequence
$(\underline{i},\underline{\alpha})=((i_0,\alpha_0), ~(i_1,\alpha_1), ...,
~(i_t,\alpha_t)) \in \mathrm{LC}_{j,\beta}$, we assume that
$$\Delta_{\underline{i},\underline{\alpha}}^C(M)= [i_0,\dots ,i_s ~|~\alpha_0
, \dots , \alpha_s](M^{(k,\eta)^+}) t_{i_{s+1},\alpha_{s+1}} \cdots
t_{i_t,\alpha_t},$$ where $(i_s,\alpha_s) < (k,\eta)^+ \leq
(i_{s+1},\alpha_{s+1})$. \\

For each $(k,\eta)$, exactly 
one of the following five cases must occur\newline
{\bf Case 1}: $(k,\eta)=(i_s,\alpha_s)$; \newline
{\bf Case 2}: $k=i_s$ and $\eta > \alpha_s$; \newline
{\bf Case 3}: $i_s < k \leq i_{s+1}$ and $\eta \in \{1,\dots
, \alpha_0\} \cup \{ \alpha_1 , \dots , \alpha_s \}$; \newline
{\bf Case 4}: $i_s < k < i_{s+1}$ and $\eta > \alpha_{s+1}$;\newline
{\bf Case 5}: $i_s < k \leq i_{s+1}$ and there exists $h \in \gc 0, s \dc$ such that $ \alpha_h < \eta <\alpha_{h+1}$.\\

We investigate each case separately to prove (ii). \\

\noindent $\bullet$ {\bf Case 1:} $(k,\eta)=(i_s,\alpha_s)$.\\ 
If $s=0$, then we have 
$$[i_0,\dots ,i_s ~|~\alpha_0 , \dots , \alpha_s](M^{(k,\eta)^+})
= t_{i_0,\alpha_0}=1 \cdot t_{i_0,\alpha_0},$$ 
and the result is proved.

Assume now that $s>0$, so the box $(i_s,\alpha_s)$ is white in $C$. 
As $(i_s,\alpha_s)=(k,\eta)>(j,\beta)=(i_0,\alpha_0)$, 
we have $t_{i_s,\alpha_s}>0$ by the primary induction hypothesis on $(j,\beta)$. 
From \cite[Proposition 3.7]{gll1}, we deduce: 
$$[i_0,\dots ,i_s ~|~\alpha_0 , \dots , \alpha_s](M^{(k,\alpha_s)^+})
= [i_0,\dots ,i_{s-1} ~|~\alpha_0 , \dots , \alpha_{s-1}](M^{(k,\alpha_s)}) 
\cdot t_{i_s , \alpha_s},$$ 
as required. \\

\noindent $\bullet$ {\bf Case 2:} $k=i_s$ and $\eta > \alpha_s$. In this case,
it follows from \cite[Proposition 3.11]{gll1} that $$[i_0,\dots ,i_s
~|~\alpha_0 , \dots , \alpha_s](M^{(k,\eta)^+})= [i_0,\dots ,i_{s} ~|~\alpha_0
, \dots , \alpha_{s}](M^{(k,\eta)}),$$ as desired.\\

\noindent $\bullet$ {\bf Case 3:} $i_s < k \leq i_{s+1}$ and $\eta \in \{1,\dots
, \alpha_0\} \cup \{ \alpha_1 , \dots , \alpha_s \}$. In this case, it follows
from \cite[Proposition 3.11]{gll1} that $$[i_0,\dots ,i_s ~|~\alpha_0 , \dots ,
\alpha_s](M^{(k,\eta)^+})= [i_0, \dots , i_{s} ~|~\alpha_0 , \dots ,
\alpha_{s}](M^{(k,\eta)}),$$ as desired.\\

\noindent $\bullet$ {\bf Case 4:} $i_s < k < i_{s+1}$ and $\eta > \alpha_{s+1}$. 

If $(k,\eta)$ is black in $C$, then $t_{k,\eta}=0$ by the  
primary induction hypothesis
on $(j,\beta)$, and so it follows from \cite[Proposition 3.11]{gll1} that
$$[i_0, \dots , i_s ~|~\alpha_0 , \dots , \alpha_s](M^{(k,\eta)}) = [i_0, \dots
, i_s ~|~\alpha_0 , \dots , \alpha_s](M^{(k,\eta)^+}),$$ as desired. \\

Now we can assume that $(k,\eta)$ is white. In this case, it follows from
\cite[Proposition 3.13]{gll1} that the difference $$[i_0, \dots , i_s ~|~\alpha_0
, \dots , \alpha_s](M^{(k,\eta)^+}) - [i_0, \dots , i_s ~|~\alpha_0 ,  \dots , 
\alpha_s](M^{(k,\eta)})$$ is a scalar multiple of $t_{k,\alpha_s}$. However,
as $(k,\eta)$ is white and the sequence $((i_0,\alpha_0), (i_1,\alpha_1),
..., (i_t,\alpha_t))$ is lacunary, we deduce from the fifth condition  of the
definition of lacunary sequences that the boxes $(k,\alpha)$ are all black for
$\alpha_0 \leq \alpha < \alpha_{s+1}$. In particular, $(k,\alpha_s)$ is black
in $C$, and so $t_{k,\alpha_s}=0$ by the  primary induction hypothesis on $(j,\beta)$.
Hence, we get $$[i_0, \dots , i_s ~|~\alpha_0 ,  \dots , 
\alpha_s](M^{(k,\eta)^+}) = [i_0 , \dots , i_s ~|~\alpha_0 ,  \dots , 
\alpha_s](M^{(k,\eta)}),$$ as required.\\

\noindent $\bullet$ {\bf Case 5:} $i_s < k \leq i_{s+1}$ and there exists $h \in \gc 0, s \dc$ such that $ \alpha_h < \eta <\alpha_{h+1}$. 
If $(k,\eta)$ is black, then $t_{k,\eta}=0$ by the induction
hypothesis on $(j,\beta)$, and so it follows from \cite[Proposition 3.11]{gll1}
that $$[i_0, \dots , i_s ~|~\alpha_0 , \dots , \alpha_s](M^{(k,\eta)}) =
[i_0, \dots , i_s ~|~\alpha_0 , \dots , \alpha_s](M^{(k,\eta)^+}),$$ as desired.
\\

From now on, we assume that $(k,\eta)$ is white, so that $t_{k,\eta}>0$ (by
the primary induction hypothesis on $(j,\beta)$). 

First, as $(k,\eta)$ is white with $k >i_h$ and $\alpha_h < \eta <
\alpha_{h+1}$, the first part of the sixth condition of a lacunary sequence is not
satisfied for the lacunary sequence $(\underline{i},\underline{\alpha})$.
Hence, the second part has to be satisfied, that is: $(i,\delta)$ is black for
all $i <i_{h+1}$ and $\alpha_h < \delta < \alpha_{h+1}$.

We claim that the sequence $((i,\eta), ~(i_{h+1},\alpha_{h+1}), \dots ,
(i_t,\alpha_t))$ is lacunary for every $i < i_{h+1}$. For, we just need to
prove that $(i',\alpha')$ is black when $i'<i_{h+1}$ and $\eta \leq \alpha ' <
\alpha_{h+1}$ as this will prove that the second parts of the fifth and sixth conditions
 are satisfied for the first step of the lacunary sequence. Under these
conditions, we have $i' < i_{h+1}$ and $\alpha_h < \eta \leq \alpha' <
\alpha_{h+1}$, so the above shows that the box $(i',\alpha')$ is indeed black.  In particular, note for later use 
that $(i,\eta)$ is black for all $i <i_{h+1}$.

 Let $i \in \{i_0, \dots, i_h\}$. Note that $(i,\eta)>
(i_0,\alpha_0)=(j,\beta)$; so (i) and (ii) apply to any 
lacunary sequence starting at  $(i,\eta)$, 
by the primary inductive hypothesis. We apply (ii) to the lacunary sequence 
$((i,\eta), (i_{h+1},\alpha_{h+1}), \dots ,
(i_t,\alpha_t))$ in two instances. Recall that $h\leq s$; so $(i_{s+1},\alpha_{s+1})$ is in this lacunary sequence. 
Moreover, as $i \leq i_s <k < i_{s+1}$, 
we have $(i_s,\alpha_s)<(k,\alpha_h)\leq (i_{s+1},\alpha_{s+1})$ if $h<s$ and 
$(i,\eta)<(k,\alpha_h)\leq (i_{s+1},\alpha_{s+1})$ if $h=s$; so, in both cases, we get
\begin{eqnarray*}
\lefteqn{
[i, i_{h+1} ,  \dots , i_t ~|~ \eta , \alpha_{h+1} , \dots , \alpha_t](M)
=}\\
&&[i , i_{h+1} ,  \dots , i_s ~|~ \eta , \alpha_{h+1} , \dots , \alpha_s]
(M^{(k,\alpha_h)}) \cdot t_{i_{s+1},\alpha_{s+1}} \cdot  ... 
\cdot t_{i_t,\alpha_t}
\end{eqnarray*}
Next, apply (ii) at the pivot $(i, \eta)$ (so we are in the case $s=-1$) 
to obtain
\[
[i, i_{h+1} ,  \dots , i_t ~|~ \eta , \alpha_{h+1} , \dots , \alpha_t](M)
=
1.t_{i,\eta} \cdot t_{i_{h+1},\alpha_{h+1}} \cdot ... 
\cdot t_{i_t,\alpha_t}
\]
Thus, 
\[
[i , i_{h+1} ,  \dots , i_s ~|~ \eta , \alpha_{h+1} , \dots , \alpha_s]
(M^{(k,\alpha_h)}) \cdot t_{i_{s+1},\alpha_{s+1}} \cdot  ... 
\cdot t_{i_t,\alpha_t}
=
t_{i,\eta} \cdot t_{i_{h+1},\alpha_{h+1}} \cdot ... 
\cdot t_{i_t,\alpha_t}
\]

 
 Recall that, for all $i \in \{i_0, \dots, i_h\}$, the box $(i,\eta)$ is black.
Hence, $t_{i,\eta}=0$ by the primary induction hypothesis on $(j,\beta)$ 
(which we can
apply as $(i,\eta) \geq (i_0,\eta) > (i_0,\alpha_0)=(j,\beta)$). 
So we get $$
[i, i_{h+1} ,   \dots , i_s ~|~ \eta , \alpha_{h+1} , \dots ,
\alpha_s](M^{(k,\alpha_h)}) \cdot t_{i_{s+1},\alpha_{s+1}} \cdot ... \cdot
t_{i_t,\alpha_t}=0.$$ Similarly, as the boxes $(i_{s+1},\alpha_{s+1})$, ...,
$(i_t,\alpha_t)$ are all white, we get that $t_{i_{s+1},\alpha_{s+1}}, \dots,
t_{i_t,\alpha_t}$ are all positive, so 
\begin{equation} 
\label{equation3}
[i, i_{h+1},  \dots , i_s ~|~ \eta , \alpha_{h+1} , \dots , \alpha_s](M^{(k,\alpha_h)}) =0, 
\end{equation}
for all $i \in \{i_0, \dots, i_h\}$.

Note now that the sequence $((i_{h+1},\alpha_{h+1}), \dots , (i_t,\alpha_t))$
is lacunary. So following the same reasoning as above we get that 
\begin{equation} 
\label{equation4}
[i_{h+1}, \dots , i_s ~|~ \alpha_{h+1} , \dots , \alpha_s](M^{(k,\alpha_h)})
=t_{i_{h+1},\alpha_{h+1}} \cdots t_{i_s,\alpha_s}>0.
\end{equation}

We then deduce from (\ref{equation3}), (\ref{equation4}) and Sylvester's 
identity that $$[i_0 , \dots i_h , i_{h+1}, \dots , i_s ~|~ \alpha_0 , \dots
,\alpha_{h-1} , \eta , \alpha_{h+1} , \dots , \alpha_s](M^{(k,\alpha_h)})=0.$$
Hence, we deduce from \cite[Proposition 3.13]{gll1} that $$[i_0, \dots , i_s
~|~\alpha_0 , \dots , \alpha_s](M^{(k,\eta)}) = [i_0, \dots , i_s ~|~\alpha_0 ,
\dots , \alpha_s](M^{(k,\eta)^+}),$$ as desired. \\

Hence, we have just proved that (ii) is true at step $(j,\beta)$. It remains
to establish (i); that is we need to prove that $t_{j,\beta} \geq 0$ with
equality if and only if $(j,\beta) \in C$. Note that this is automatic under assumption
(a) thanks to Theorem~\ref{thm:cell2}. 
Under assumption (b), 
by (ii) at step $(j,\beta)$ with $(k,\eta)=(j,\beta)$, we have
$$\Delta_{j,\beta}^C(M)=t_{j,\beta} \cdot t_{j_1,\beta_1} \cdot ... \cdot
t_{j_d,\beta_d}.$$ As the boxes $(j_l,\beta_l)$ with $1 \leq l \leq d$ are
white, we deduce from the induction hypothesis that $t_{j_l,\beta_l}>0$ for
all $l$. As $\Delta_{j,\beta}^C(M) \geq 0$ with equality exactly when
$(j,\beta) \in C$ (by hypothesis), we get that $t_{j,\beta} \geq 0$ with
equality exactly when $(j,\beta) \in C$. This proves (i).
\end{proof}


The following two results now follow easily.


\begin{proposition}\label{proposition-test1} 
Let $M \in \mc_{m,p}(\bbR)$. We keep the notation of
Proposition~\ref{proposition:a-or-b}. Assume that $M$ is tnn and belongs to the tnn cell  $S^0_C$ 
associated to $C$. Let $(j,\beta)\in \gc 1,m \dc \times \gc 1,p \dc$. 

Then we have $\Delta_{j,\beta}^C(M) =0$ if $(j,\beta) \in C$ and
$\Delta_{j,\beta}^C(M) > 0$ if $(j,\beta) \notin C$. 
\end{proposition}

\begin{proof} 
Note that $\Delta_{j,\beta}^C(M) \geq 0$ as $M$ is tnn. This
result follows easily from the previous proposition. Indeed,
$\Delta_{j,\beta}^C(M) = t_{j_0,\beta_0} \cdot t_{j_1,\beta_1} \cdots
t_{j_d,\beta_d}$. As the boxes $(j_s,\beta_s)$ for $s>0$ are white in $C$ and
$M$ belongs to the tnn cell $S^0_C$  associated to $C$, we have $t_{j_s,\beta_s}>0$ for
all $s>0$ by Theorem~\ref{thm:cell2}. Hence, $\Delta_{j,\beta}^C(M) =0$ if and only if
$t_{j,\beta}=t_{j_0,\beta_0}=0$; that is, if and only if $(j,\beta) \in C$ as
$M$ belongs to the tnn cell associated to $C$. 
\end{proof}

\begin{proposition}\label{proposition-test2} 
Let $M \in \mc_{m,p}(\bbR)$. We keep the notation of 
Proposition~\ref{proposition:a-or-b}.
Assume that, for all $(j,\beta)$, we have $\Delta_{j,\beta}^C(M) =0$ if
$(j,\beta) \in C$ and $\Delta_{j,\beta}^C(M) > 0$ if $(j,\beta) \notin C$.
Then $M$ is tnn and belongs to the tnn cell  $S^0_C$ associated to $C$. 
\end{proposition}

\begin{proof}
By Proposition~\ref{proposition:a-or-b}, we know that each $t_{i,\alpha}\geq 0$ 
and that $t_{i,\alpha}> 0$ precisely when $(i,\alpha)\in C$.
The result now follows from Theorem \ref{thm:cell2}.
\end{proof} 

Propositions~\ref{proposition-test1}
and~\ref{proposition-test2} together provide a test for membership of a fixed
totally nonnegative cell. Each cell is specified by a Cauchon diagram. For a
given $m\times p$ Cauchon diagram $C$ assume that a lacunary sequence has been
chosen for each box $(j,\beta)$ and for a given real $m\times p$ matrix $M$
check whether $\Delta_{j,\beta}^C(M) > 0$ for each white box and
$\Delta_{j,\beta}^C(M) =0$ for each black box. Then $M$ is totally nonnegative
and is in the cell specifed by $C$ if and only if these $mp$ statements are
all true. This is summarised in our main result. 

\begin{theorem} \label{theorem-main} 
Let $C$ be an $m \times p$ Cauchon diagram. For all $(j,\beta) \in \gc 1,m\dc \times \gc 1,p \dc$, fix a
lacunary sequence $(\underline{j},\underline{\beta})=((j_0,\beta_0),
~(j_1,\beta_1), ..., ~(j_d,\beta_d)) \in \mathrm{LC}_{j,\beta}$, and set
$\Delta_{j,\beta}^C := \Delta_{\underline{j},\underline{\beta}}^C$. Let $M \in \mathcal{M}_{m,p}(\mathbb{R})$. The following are equivalent.
\begin{enumerate}
\item $M$ is tnn and belongs to the tnn cell  $S^0_C$ associated to $C$. 
\item For all $(j,\beta) \in \gc 1, m \dc \times \gc 1, p \dc$, we have $\Delta_{j,\beta}^C(M) =0$ if $(j,\beta) \in C$ and  $\Delta_{j,\beta}^C(M) > 0$ if $(j,\beta) \notin C$.
\end{enumerate}
\end{theorem}

Note that this test only involves $mp$ minors. In the case that $C$ is the
Cauchon diagram with all boxes coloured white, the test states that a real
matrix $M$ is totally positive if and only if each final minor of $M$ is
strictly positive. As discussed earlier, this is the well-known Gasca and
P\~ena test, but applied to final minors rather than initial minors.

\begin{example} Consider the following Cauchon diagram:
\begin{center}
\begin{tikzpicture}[xscale=0.6, yscale=0.6]
\node at (0,1.5) {$C=$};
\draw [fill=gray] (2.1,1.1) rectangle (2.9,1.9);
\draw [fill=gray] (3.1,1.1) rectangle (3.9,1.9);
\draw [fill=gray] (4.1,2.1) rectangle (4.9,2.9);
\draw [help lines, black, thick] (2,0) grid   (5,3);
\end{tikzpicture}
\end{center}
Then a $3 \times 3$ real matrix $M$ is tnn and belongs to the cell associated
to $C$ if and only if 
$$\begin{array}{lll} \Delta_{1,1}(M)=[13|12](M) >0 &
\Delta_{1,2}(M)=[12|23](M)>0&
\Delta_{1,3}(M)=[1|3](M) =0 \\ & & \\ \Delta_{2,1}(M)=[23|12](M) =0 &
\Delta_{2,2}(M)=[23|23](M) =0 & \Delta_{2,3}(M)=[2|3](M) >0 \\ & & \\
\Delta_{3,1}(M)=[3|1](M) >0~ & \Delta_{3,2}(M)=[3|2](M) >0~ &
\Delta_{3,3}(M)=[3|3](M) >0. \\
\end{array}$$
(The lacunary sequences have all been obtained 
by using Algorithm \ref{algo1}.)

It is easy to check that the matrix $M=\left( \begin{array}{ccc} 16 & 5 & 0 \\ 12 & 6 & 3 \\ 4 & 2 & 1 \end{array}\right)$ satisfies the above nine conditions. Hence, we deduce from Theorem \ref{theorem-main} that $M$ is tnn and belongs to the tnn cell $S^0_C$ associated to $C$.
\end{example}


\begin{minipage}[c]{\linewidth}
~\\
\noindent 
St\'ephane Launois\\ 
School of Mathematics, Statistics \& Actuarial Science,\\ 
University of Kent, \\
Canterbury, Kent CT2 7NF, United Kingdom\\
E-mail: {\tt S.Launois@kent.ac.uk}
 \\[10pt]
Tom Lenagan\\
Maxwell Institute for Mathematical Sciences,\\
School of Mathematics, University of Edinburgh,\\
James Clerk Maxwell Building, King's Buildings, Mayfield Road,\\
Edinburgh EH9 3JZ, Scotland, UK\\
E-mail: {\tt tom@maths.ed.ac.uk}
\end{minipage}


\end{document}